\documentclass[conference]{IEEEtran}
\IEEEoverridecommandlockouts
  \usepackage[noadjust]{cite}
  \usepackage{amsmath}
  \usepackage{amsthm}
  \usepackage{graphicx}
  \usepackage{booktabs}
  \usepackage{threeparttable}
  \newcommand{\uplevel}{IEH }
  
  \renewcommand{\vec}[1]{\boldsymbol{#1}} 
  \newcommand{\tc}{t_{\rm c}}
  \newcommand{\te}{t_{\rm e}}
  \ifCLASSINFOpdf
  \else
  \fi  
  
\begin{document} 
  
  \title{Autonomous-Collaborative Energy Optimization\\ for Interconnected Energy Hubs \\Based on Transactive Control}
  
          
  
  
\author{\IEEEauthorblockN{Yizhi Cheng}
\IEEEauthorblockA{Department of Electrical Engineering\\
Shanghai Jiao Tong University\\
Shanghai, China\\
Email: tabdel@sjtu.edu.cn}
\thanks{This work was supported by the National Key R\&D Program of China (2018YFB0905000)} 
\and
\IEEEauthorblockN{Peichao Zhang}
\IEEEauthorblockA{Department of Electrical Engineering\\
Shanghai Jiao Tong University\\
Shanghai, China\\
Email: pczhang@sjtu.edu.cn}}

  \maketitle
  
  \begin{abstract}
  Motivated by various benefits of multi-energy integration, this paper establishes a bi-level framework based on transactive control to realize energy optimization among multiple interconnected energy hubs (EHs). A storage-energy-equivalent method as well as its mathematical proof are provided in the lower level to realize nonlinear constraints relaxation of EH model, while the upper level solves the collaborative problem iteratively in both day-ahead and real-time stages. The proposed method can preserve information privacy and operation authority of each EH while satisfying real-time control requirement, and its effectiveness has also been verified by a simulation case. 
  \end{abstract}
  

  \IEEEpeerreviewmaketitle

  \section{Introduction}
  
  \IEEEPARstart Continuous environment deterioration and energy depletion have necessitated the comprehensive utilization of various forms of energy in modern life. 
   It is believed that the development of relevant technologies would spur the advent of integrated energy service companies that manage several types of energy concurrently. Despite that multi-energy integration might improve overall energy efficiency at lower costs, complexities are also introduced to the system with regards to operation and management. To solve such problems, researchers from ETH Zurich, Switzerland has proposed concept of energy hub (EH) to investigate the energy system as a whole, whose path has been followed by a number of researches\cite{7842813}. Recent years have also witnessed a research re-orientation from energy optimization of a standalone EH to a collaborative optimization among multiple interconnected EHs (IEH)\cite{EN2448}.
  
  Transactive control (TC) developed in recent years is a set of economic and control mechanisms that allows the dynamic balance of supply and demand across the entire electrical infrastructure using value as a key operational parameter\cite{HU2017}. Although many studies have used TC in energy management, according to the authors' knowledge, to date few researches have applied the TC framework to IEH energy optimization. 
  
  This paper assumes an \uplevel that is connected to main grid and natural gas network and can provide energy services for its inferior sectors EHs.
  At the same time, the \uplevel agent serves as an interface between lower-level EHs and the main grid while responding to dispatching signals from upstream grid.
  This paper then proposes an autonomous-collaborative optimization framework based on TC to coordinate energy management for IEH where incentive and responsive signals are exchanged back and forth between \uplevel agent and EH. The main contributions of this paper include:
  \begin{itemize}
    \item it establishes a bi-level framework to coordinate energy management among interconnected EHs based on TC;
    \item it proposes a storage-energy-equivalent method with its proof to model the lower-level problem in a convex way;
    \item it simplifies dual variables of the upper-level problem and further solves the problem by bisection method iteratively to meet real-time control requirements.
  \end{itemize}
  
  \section{Lower-level: Autonomous optimization based on storage-energy-equivalent method}
  
  \subsection{EH model}
  
  
  Fig. \ref{fig_EH_model} shows the structure of an EH consists of combined heat and power(CHP) plant, natural gas furnace(GF), electric energy storage(EES) and thermal energy storage(TES). Its model is:
  \begin{equation}
    \begin{split}
    \label{eqn_IESmodel}
    \left(
      \begin{matrix}
        \eta \rm_{ee} & \eta \rm _{ge}^{CHP} & 0 \\ 
        0 & \eta {\rm _{gth}^{CHP}} & \eta \rm{_{gth}^{GF}}
      \end{matrix}
    \right)
    \left(
      \begin{matrix}
        P_{\mathrm{e},t} \\ G_{\mathrm{g},t}^{\rm CHP} \\ G_{\mathrm{g},t}^{\rm GF}
      \end{matrix}
    \right)
    + 
    \left(
      \begin{matrix}
        P_{\mathrm{dch},t}^{\rm EES} - P_{\mathrm{ch},t}^{\rm EES} \\
        H_{\mathrm{dch},t}^{\rm TES} - H_{\mathrm{ch},t}^{\rm TES}
      \end{matrix}
    \right)\\
    +
    \left(
      \begin{matrix}
        P_t^{\rm RES} \\ 0
      \end{matrix}
    \right)
    -
    \left(
      \begin{matrix}
        P_t^{\rm curt} \\ H_t^{\rm curt}
      \end{matrix}
    \right)
    =
    \left(
      \begin{matrix}
        L_{\mathrm{e},t}^{\rm sl} \\ L_{\mathrm{th},t}^{\rm sl}
      \end{matrix}
    \right)
    +
    \left(
      \begin{matrix}
        L_{\mathrm{e},t} \\ L_{\mathrm{th},t}
      \end{matrix}
    \right),\forall t
    \end{split}
  \end{equation}
  where $P_{\mathrm{e},t}$ denotes the electricity power EH imports from the main gird at time ${t}$, and $P_{\mathrm{e},t}<0$ implies that the EH sells surplus electricity to the gird.
  $G_{\mathrm{g},t}^{\rm CHP}$ and $G_{\mathrm{g},t}^{\rm GF}$ denote natural gas consumed by CHP and GF, respectively.
  $\eta _{ee}$ denotes transmission efficiency.
  $\eta \rm _{ge}^{CHP}$ and $\eta \rm _{gth}^{CHP}$ are the gas-electric and gas-thermal efficiencies of CHP, and $\eta \rm_{gth}^{GF}$ is the efficiency of GF. 
  $P_t^{\rm RES}$ demonstrates the power generated by renewable energies under the maximum power point tracking mode.
  $P_{\mathrm{ch},t}^{\rm EES}$ and $P_{\mathrm{dch},t}^{\rm EES}$ represents the charging and discharging power of EES, respectively, while $H_{\mathrm{ch},t}^{\rm TES}$ and $H_{\mathrm{ch},t}^{\rm TES}$ are the charging and discharging amount of TES.
  $ L_{\mathrm{e},t}^{\rm sl}$ and $ L_{\mathrm{th},t}^{\rm sl}$ denote the shiftable electric load and thermal load, while $L_{\mathrm{e},t}$ and $L_{\mathrm{th},t}$ denote corresponding non-shiftable loads.
  $P_t^{\rm curt}$ and $ H_t^{\rm curt}$ demonstrate the curtailed renewable energies and heat.
  \begin{figure}
    \centering
    \includegraphics[width=0.48\textwidth]{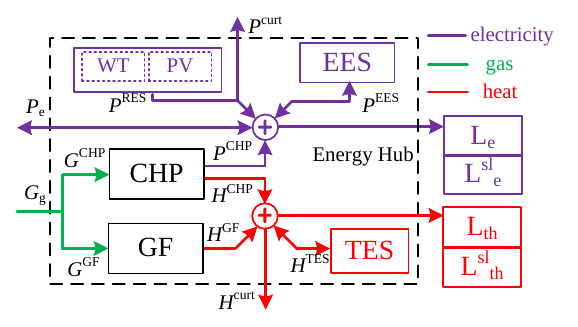}
    \caption{Schematic of a typical integrated energy system}
    \label{fig_EH_model}
  \end{figure}
  \subsection{Problem formulation}
  The operating cost during period $t$ can be split into two parts: electricity purchasing cost and gas purchasing cost.
  \begin{equation}
    F_t = \mu _{\mathrm{e},t} P_{\mathrm{e},t} +
    \mu _{\mathrm{g},t}
    \left(G_{\mathrm{g},t}^{\rm CHP}
    + G_{\mathrm{g},t}^{\rm GF}\right)
  \end{equation}
  where $\mu _{\mathrm{e},t}$ and $\mu _{\mathrm{g},t}$ denote the utility electricity price and natural gas price at time $t$, respectively.
  
  During every scheduling period, after updating forecasted local renewable energies output and load demand, each EH seeks to minimize the expected costs across the remaining periods in an autonomous manner.

  In addition to power balance constraints\eqref{eqn_IESmodel} and constraints (4-16) in \cite{8245724}, constraints also include limits of charging and discharging power:
  \begin{equation}
    \label{eqn_EES}
    P_{\mathrm{ch},t}^{\rm EES}P_{\mathrm{dch},t}^{\rm EES}=0,\;\;
    H_{\mathrm{ch},t}^{\rm TES}H_{\mathrm{dch},t}^{\rm TES}=0, \forall t
  \end{equation}
  constraints associated with shiftable loads:
  \begin{align}
    \label{eqn_shiftable_load}
    \begin{split}
      L_{\mathrm{e},t}^{\rm sl} \geq 0,\;& L_{\mathrm{th},t}^{\rm sl} \geq 0 ,\forall t \\
      \sum_{t=\tc}^{\te}{L_{\mathrm{e},t}^{\rm sl}}=L_{\rm e}^{\rm sl},\;&
      \sum_{t=\tc}^{\te}{L_{\mathrm{th},t}^{\rm sl}}=L_{\rm th}^{\rm sl}
    \end{split}
  \end{align}
  and upper and lower limits of energy curtailment:
  \begin{equation}
    \label{eqn_curtailment}
    0\leq P_t^{\rm curt} \leq P_t^{\rm RES},\;\; H_t^{\rm curt}>0,\forall t
  \end{equation}
  
  %
  
  The autonomous optimization problem can thus be formulated as:
  \begin{align}
    \label{eqn_auto_optimization}
    \begin{split}
      &\min \sum_{t=\tc}^{\te}F_t\\
      \mathrm{s.t.} 
      \eqref{eqn_IESmodel}
      \eqref{eqn_EES}
      \eqref{eqn_shiftable_load}
      \eqref{eqn_curtailment},
      &\text{constraints in \cite{8245724}}
    \end{split}
    \tag{P1}
  \end{align}
  where $\tc$ and $\te$ denote current and end time index. 
  
  \subsection{Storage-energy-equivalent method}
  Some researchers have introduced binary variables\cite{MAJIDI2017157} to remove bi-linear terms in constraint\eqref{eqn_EES}, while this section proposes a method to relax this constraint and turn the autonomous optimization into a convex one. In this case, \eqref{eqn_auto_optimization} then can be represented as:
  \begin{align}
    \label{eqn_auto_optimization_relaxed}
    \begin{split}
      &\min \sum_{t=\tc}^{\te}F_t\\
      \mathrm{s.t.} \;\;
      \eqref{eqn_IESmodel}
      &\eqref{eqn_shiftable_load}
      \eqref{eqn_curtailment},
      \text{constraints in \cite{8245724}}
    \end{split}
    \tag{P2}
  \end{align}
  
  The relationship between optimal solutions of problem \eqref{eqn_auto_optimization} and \eqref{eqn_auto_optimization_relaxed} is discussed below. Without loss of generality, this paper only discusses the mutual exclusiveness of charging/discharging mode of EES, and the same method is also applicable to TES.

  Let the feasible regions of \eqref{eqn_auto_optimization} and \eqref{eqn_auto_optimization_relaxed} be $K_1$, $K_2$, optimal solutions be $ x_1^*$, $x_2^*$, and optimal values be $f(x_1^*)$,$f(x_2^*)$, respectively. Let $P_{\mathrm{ch},t}^{\rm EES*}$ and $P_{\mathrm{dch},t}^{\rm EES*}$ denote the optimal charging/discharging power of EES in $x_2^*$ at time $t$. Then net energy change of EES during this period is:
  \begin{equation}
    \label{eqn_EES_delta_energy}
    \Delta S_t^{\rm EES*}=P_{\mathrm{ch},t}^{\rm EES*}\eta_{\rm ch}^{\rm EES}-
    \frac{P_{\mathrm{dch},t}^{\rm EES*}}{\eta_{\rm dch}^{\rm EES}}
  \end{equation}
  where $\eta_{\rm ch}^{\rm EES}$ and $\eta_{\rm dch}^{\rm EES}$ denote charging and discharging efficiencies of EES, respectively.

  The transformation method first calculates a pair of charging/discharging power value that leads to equivalent energy change during the control time for EES:
  \begin{equation}
    \label{eqn_transform}
    \left(
      \widetilde P_{\mathrm{ch},t}^{\rm EES*}, \widetilde P_{\mathrm{dch},t}^{\rm EES*}
    \right)
    =\begin{cases}
      (\Delta S_t^{\rm EES*}/\eta_{\rm ch}^{\rm EES}, 0),\text{if} \Delta S_t^{\rm EES*}\geq 0 \\
      (0,\; -\Delta S_t^{\rm EES*}\eta_{\rm dch}^{\rm EES}), \text{if} \Delta S_t^{\rm EES*}<0
    \end{cases}
  \end{equation}
  
  Then, for $\forall t$, modify the optimal EES power from $P_{\mathrm{ch},t}^{\rm EES*}$ and $P_{\mathrm{dch},t}^{\rm EES*}$ to $\widetilde P_{\mathrm{ch},t}^{\rm EES*}$ and $\widetilde P_{\mathrm{dch},t}^{\rm EES*}$, respectively. Besides, modify curtailment variable $P_t^{\rm curt*}$ as well to maintain power balance after the transformation:
  \begin{equation}
    \label{eqn_delta_curtailment}
      \widetilde P_t^{\rm curt*} = P_t^{\rm curt*} + \Delta P_{\mathrm{dch},t}^{\rm EES}
  \end{equation}
  where $\Delta P_{\mathrm{dch},t}^{\rm EES}$
  can be calculated as:
  \begin{equation}
    \label{eqn_EES_delta_power}
        \Delta P_{\mathrm{dch},t}^{\rm EES} =
        \left(
          \widetilde P_{\mathrm{dch},t}^{\rm EES*} - \widetilde P_{\mathrm{ch},t}^{\rm EES*}
          \right) - 
          \left(
            P_{\mathrm{dch},t}^{\rm EES*} - P_{\mathrm{ch},t}^{\rm EES*}
            \right)
  \end{equation}
  
  Let $\widetilde x_2^*$ denotes the new vector after the modification.

  The following theorem provides a sufficient condition for equivalency of these two optimal solutions:
  \newtheorem{theorem}{Theorem}
  \begin{theorem}
    \label{thm_normal_mode}
    If no renewable energies need to be curtailed, then $P_{\mathrm{ch},t}^{\rm EES*}P_{\mathrm{dch},t}^{\rm EES*} = 0, \forall t$ holds.
  \end{theorem}
  \begin{proof}
    To prove it by contradiction, suppose that $\exists t\in \left[\tc,\te\right]$, such that $P_{\mathrm{ch},t}^{\rm EES*} > 0,\;P_{\mathrm{dch},t}^{\rm EES*} > 0$. Without loss of generality, assume that $\Delta S_t^{\rm EES*} \geq 0$, and similar proof can be derived when $\Delta S_t^{\rm EES*} < 0$. According to \eqref{eqn_transform}, the net discharge power change \eqref{eqn_EES_delta_power} can be simplified as:
    \begin{equation}
      \label{eqn_EES_delta_power_geq0}
      \Delta P_{\mathrm{dch},t}^{\rm EES} =\left(
        \frac{1}{\eta_{\rm dch}^{\rm EES}\eta_{\rm ch}^{\rm EES}}-1
      \right)
      P_{\mathrm{dch},t}^{\rm EES*} > 0
    \end{equation}
  which means that compared with $x_2^*$, EES in $\widetilde x_2^*$ consumes less power. Since there is no renewable energies curtailment, according to \eqref{eqn_IESmodel}, the \uplevel agent could purchase less electricity from the main grid to supply loads, thus reducing the overall costs, which contradicts with the fact that $x_2^*$ is defined as an optimal solution of \eqref{eqn_auto_optimization_relaxed}.
  \end{proof}
  
  According to Theorem \ref{thm_normal_mode}, without energy curtailment the model has no incentive to charge and discharge simultaneously\cite{BECK2016331}. Actually, references\cite{BECK2016331, 8486723} have conducted similar relaxation steps under the precondition that this sufficient condition is always established. However, when renewable energies are abundant compared with the local load level, surplus power should be curtailed and this precondition is not established. In this case, an augmented sufficient condition is given in the following theorem:
  \begin{theorem}
    \label{thm_extreme_mode}
    If optimal solution of \eqref{eqn_auto_optimization_relaxed} $x_2^*$ satisfies:
    \begin{equation}
      \label{eqn_condition_constraint}
      \frac{P_t^{\rm RES} - P_t^{\rm curt*}}{1-\eta \rm _{ch}^{EES}\eta \rm _{dch}^{EES}} \geq \min (P_{\mathrm{ch},t}^{\rm EES*},\frac{P_{\mathrm{dch},t}^{\rm EES*}}{\eta \rm _{ch}^{EES}\eta \rm _{dch}^{EES}})
    \end{equation}
    then the new vector $\widetilde x_2^*$ is an optimal solution of \eqref{eqn_auto_optimization}.
  \end{theorem}
  \begin{proof}
  Condition 1: When $x_2^* \in K_1$. First, since \eqref{eqn_EES} is satisfied and the right-hand-side of \eqref{eqn_condition_constraint} equals zero, condition \eqref{eqn_condition_constraint} is always met. Second, on the one hand, because $x_2^*$ and $x_1^*$ denote the feasible and optimal solution of \eqref{eqn_auto_optimization} respectively, it can be derived that $f(x_1^*) \leq f(x_2^*)$. On the other hand, since $K_1 \subset K_2$, then $f(x_1^*) \geq f(x_2^*)$. Therefore, $f(x_1^*)=f(x_2^*)$ and $x_2^*$ is also an optimal solution of \eqref{eqn_auto_optimization}. At last, it is obvious that $x_2^* = \widetilde x_2^*$. To sum up, $\widetilde x_2^*$ is an optimal solution of \eqref{eqn_auto_optimization}.
  
  Condition 2: When $x_2^* \notin K_1$, which means that $\exists t\in [\tc,\te]$, such that $P_{\mathrm{ch},t}^{\rm EES*}P_{\mathrm{dch},t}^{\rm EES*} > 0$. 
  Let's first verify that  $\widetilde x_2^*$ still satisfy constraints associated with modified variables,i.e. $P_{\mathrm{ch},t}^{\rm EES*}, P_{\mathrm{dch},t}^{\rm EES*}$ and $P_t^{\rm curt*}$. It is obvious that $\widetilde x_2^*$ satisfies \eqref{eqn_IESmodel}\eqref{eqn_EES} already. Since state-of-charge (SOC) change of EES stays unchanged, all SOC-related constraints still hold. Besides, upper and lower limits of EES charging/discharging power are satisfied according to \eqref{eqn_transform}. At last, substitute  \eqref{eqn_EES_delta_power_geq0}, \eqref{eqn_condition_constraint} into \eqref{eqn_delta_curtailment}, and it can be proved that constraints\eqref{eqn_curtailment} are satisfied.
  %
  Therefore, all constraints of \eqref{eqn_auto_optimization} are met for $\widetilde x_2^*$, and $\widetilde x_2^*$ is a feasible solution of \eqref{eqn_auto_optimization}. Since $f(x_2^*) = f(\widetilde x_2^*)$, hereafter we apply conclusion of condition 1, and it can be finally derived that $f(x_1^*) = f(\widetilde x_2^*)$. Thus, $\widetilde x_2^*$ is an optimal solution of \eqref{eqn_auto_optimization}.
  \end{proof}
  It should be pointed out that 
  extreme circumstance when \eqref{eqn_condition_constraint} is unsatisfied never occurs due to the optimal planning procedure of EH in practical operations. This fact is also verified in the simulation case. As a consequence, by solving \eqref{eqn_auto_optimization_relaxed}, an optimal solution of problem \eqref{eqn_auto_optimization} can be obtained. Therefore, in the subsequent model, each EH autonomously optimizes according to \eqref{eqn_auto_optimization_relaxed}. 
  
  \section{Upper-level: Collaborative optimization based on transactive control}
  \begin{figure*}
    \centering
    \includegraphics[width=0.9\textwidth]{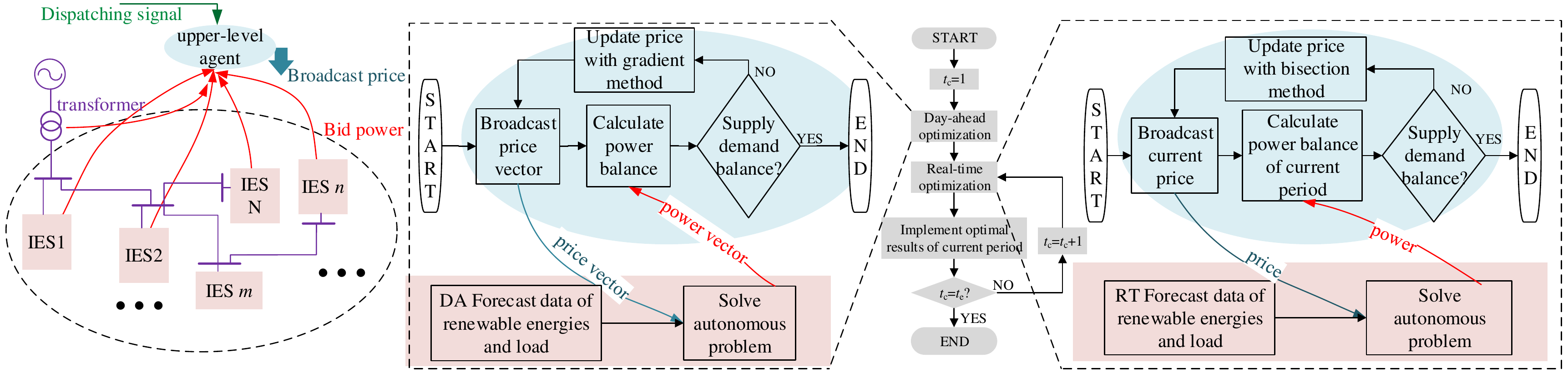}
    \caption{Operation framework of day-ahead and real-time optimization}
    \label{fig_framework}
  \end{figure*}
  
  \subsection{Problem formulation and decomposition}
  The structure of IEH is shown in Fig. \ref{fig_framework}(a). During each period, upper-level \uplevel agent aims at minimizing its overall costs across the remaining periods while balancing the supply and demand, and limiting the transformer capacity. The optimization problem can be modeled as follows:
  \begin{align}
    \label{eqn_col_optimization}
    \begin{split}
      &\;\min \sum_{n=1}^{N}\sum_{t=\tc}^{\te}F_{t,n}\\
      \mathrm{s.t.}&{\sum_{n=1}^{N}P_{\mathrm{e},t,n}}=P_{\mathrm{e},t}^{\rm Tr},\forall t\\
      &-P_{\mathrm{e},t}^{\rm Tr,out,max}\leq P_{\mathrm{e},t}^{\rm Tr} \leq P_{\mathrm{e},t}^{\rm Tr, in,max},\forall t\\
      &\text{
      \eqref{eqn_IESmodel}$_n$
      \eqref{eqn_shiftable_load}$_n$
      \eqref{eqn_curtailment}$_n$,
      constraints in \cite{8245724}}, \forall n
    \end{split}
    \tag{P3}
  \end{align}
  where $n$ is the EH index. $P_{\mathrm{e},t}^{\rm Tr}$ is the electricity that transformer imports from the main grid. $P_{\mathrm{e},t}^{\rm Tr, in,max},P_{\mathrm{e},t}^{\rm Tr,out,max}> 0$ are the maximum power exchanged with the main grid.
  
  The upper level problem\eqref{eqn_col_optimization} should have been solved in an absolute centralized manner after gathering all EHs' detailed information. However, to preserve information privacy, this paper advocates to solve it in a distributed way by employing Lagrange dual decomposition method and transactive control. 
  The Lagrangian relaxed dual problem is: 
  \begin{align}
    \label{eqn_dual_problem}
    \begin{split}
    &\;\;\;\;\max_{\forall t,\lambda _t} \varphi (\lambda_t)= \max \mathrm{inf} L\\
    &\mathrm{s.t.} -P_{\mathrm{e},t}^{\rm Tr,out,max}\leq P_{\mathrm{e},t}^{\rm Tr} \leq P_{\mathrm{e},t}^{\rm Tr, in,max},\forall t\\
    &\text{
      \eqref{eqn_IESmodel}$_n$
      \eqref{eqn_shiftable_load}$_n$
      \eqref{eqn_curtailment}$_n$,
      constraints in \cite{8245724}}, \forall n
    \end{split}
    \tag{P4}
  \end{align}
  where $L$ is the Lagrangian relaxation function after introducing the Lagrange multipliers $\lambda_{\tc},\lambda_{t_{\rm c}+1},...,\lambda_{\te}$ associated with power balancing constraints:
  \begin{equation}
    \label{eqn_Lagrange_function}
    L={\sum_{n=1}^{N}\sum_{t=\tc}^{\te}F_{t,n}}
    +\sum_{t=\tc}^{\te}\lambda _t{(\sum_{n=1}^{N}P_{\mathrm{e},t,n}}-P{_{\mathrm{e},t}^{\rm Tr}})
  \end{equation}
  
  Since the primal problem \eqref{eqn_col_optimization} is linear, strong duality theorem holds and the optimal value of \eqref{eqn_dual_problem} is equivalent to \eqref{eqn_col_optimization}.
  
  For $\forall t$, define local electricity price $\lambda _{\mathrm{e},t}=\mu _{\mathrm{e},t} + \lambda _t$,
  then problem \eqref{eqn_dual_problem} is decomposed into one master problem where the \uplevel agent adjusts local price vector 
  $\Lambda_{\tc}=\left\{\lambda_{\mathrm{e},\tc},\lambda_{\mathrm{e},t_{\rm c}+1},...,\lambda_{\mathrm{e},\te}\right\}$ to strike a general supply and demand balance, in addition to $N+1$ subproblems where each EH minimizes its cost under the local electricity price $\Lambda_{\tc}$ and the transformer itself maximizes its profits through exchanging electricity power with the main gird.
  
  \subsection{Two-stage procedure}
  A market is established in the upper level where each EH participates as independent energy demander and the transformer participates as the supplier. This paper then proposes a two-stage procedure to obtain optimal result as well as to meet the real-time requirement.

  \subsubsection{Day-ahead (DA) optimization}
  The process is illustrated in Fig. \ref{fig_framework}(b). Gradient method is adopted during day-ahead stage to solve the master problem iteratively.
  
  Assume that the forecasted local price vector of $k$-th iteration is:
  \begin{equation}
    \label{eqn_forecast_price}
    \hat{\vec\Lambda} ^k = \left\{
      \hat\lambda _{\rm e,1}^k,
      \hat\lambda _{\rm e,2}^k,
      \dots,
      \hat\lambda _{\mathrm{e},\tc}^k,
      \dots,
      \hat\lambda _{\mathrm{e},\te}^k
      \right\}
  \end{equation}
  where $\hat\lambda _{\mathrm{e},t}^k$ denotes the forecasted local price at time $t$ after $k$-th iteration.
  
  At each iteration, each EH solves subproblem according to the broadcasted price $\hat{\vec\Lambda} ^k$ and bids the optimal power vector $\vec{P}_{\mathrm{e},n}^{*}$ to the market. After receiving all bidding data, the \uplevel agent obtains optimal transformer power vector $\vec{P}_{\mathrm{e}}^{\rm Tr*}$ and computes the balance vector $\Delta \vec{P} = -\vec{P}_{\mathrm{e}}^{\rm Tr*} + \sum_{1}^{N}{\vec{P}_{\mathrm{e},n}^{*}}$. The \uplevel agent then updates and broadcasts the price vector $\hat{\vec\Lambda} ^{k+1} = \hat{\vec\Lambda} ^k+\eta^k\Delta \vec{P}$ where $\eta^k$ denotes a feasible step length. These steps are iteratively repeated until supply and demand balance is achieved.
  \subsubsection{Real-time (RT) optimization}
  The process is illustrated in Fig. \ref{fig_framework}(c). During every control period, the \uplevel agent endeavors to eliminate the impacts of multiple uncertainties through rolling horizon optimization. Besides, at time $\tc$, local electricity prices at time {$t_{\rm c}+1,t_{\rm c}+2,\dots,\te$} are assumed to be equal to the DA forecasted prices. 
  Let $\lambda\rm _e^{max}$ and $\lambda\rm _e^{min}$ denote the upper and lower bound of local electricity price, and the processes of real-time optimization are described as follows:
  \\S0: The \uplevel agent broadcasts the DA forecast price vector:
  $\hat{\vec\Lambda} = \left\{
    \hat\lambda _{\rm e,1},
    \hat\lambda _{\rm e,2},
    \dots,
    \hat\lambda _{\mathrm{e},\tc},
    \dots,
    \hat\lambda _{\mathrm{e},\te}
    \right\}$;
  \\S1: The \uplevel agent broadcasts price $\lambda_{\mathrm{e},\tc}^p$ of $\tc$, where $p$ is real-time iteration index. ($\lambda_{\mathrm{e},\tc}^0 = \lambda\rm _e^{max}$ and $\lambda_{\mathrm{e},\tc}^1 = \lambda\rm _e^{min}$).
  \\S2: Each EH generates the $p$-th price vector:
  \begin{equation}
    \label{eqn_RT_price}
    \vec\Lambda_{\tc}^p = \left\{
      \lambda_{\mathrm{e},\tc}^p,
      \overbrace{\hat\lambda _{\mathrm{e},t_{\rm c}+1},
      \dots,
      \hat\lambda _{\mathrm{e},\te}}^\text{
        day-ahead forecast price
      }
      \right\}
  \end{equation}
  
  The EH then solves the autonomous subproblems and bids $P_{\mathrm{e},\tc,n}^{p*}$, optimal power of current period, to the market.
  \\S3: The \uplevel agent obtains $P_{\mathrm{e},\tc}^{\mathrm{Tr},p*}$, the optimal transformer power of current period, and calculates the overall power balance:
  \begin{equation}
    \label{eqn_balance}
    \Delta \vec{P}_{\mathrm{e},\tc}^p =
     -P_{\mathrm{e},\tc}^{\mathrm{Tr},p*}
     +\sum_{n=1}^{N}{P_{\mathrm{e},\tc,n}^{p*}}
  \end{equation}
  \\S4: If power balance $\Delta \vec{P}_{\mathrm{e},\tc}^p$ equals zero, then set clearing price $\lambda_{\mathrm{e},\tc}^*$ to be $\lambda_{\mathrm{e},\tc}^p$ and step into S5. Else, the \uplevel agent uses bisection method to update price according to $\Delta \vec{P}_{\mathrm{e},\tc}^p$, and step back to S1.
  \\S5: Each EH implements optimal result and moves into period $t_{\rm c}+1$.
  
  It should be noted that both DA and RT stage require multiple iterations to solve the optimization master problem. Since relatively accurate price vector has been forecasted in DA stage, the RT schedule transforms high-dimensional local prices to single-dimensional ones, thus making it feasible to update $\lambda_{\mathrm{e},\tc}^{p+1}$ by the bisection method, which outperforms gradient method with respect to convergence speed.
  
  \section{Case Study}
  \subsection{Simulation Setup}
  An IEH consisting of 15 EHs is illustrated in the simulation case. Parameters of EHs, together with load and renewable energies data are listed in \cite{8245724}.
  The interval period is 1h, maximum and minimum electricity price in the market is 1.5 and 0 yuan/kWh, respectively. The main grid follows real-time utility electricity prices. The price of natural gas is 3.3 yuan/m$^3$. Its density is 0.79kg/m$^3$, and the calorific value is 45MJ/kg. The day-ahead, intra-day, and real-time forecast errors of renewable energies are $\pm 30\%$, $\pm 10\%$, and $\pm 5\%$, respectively; the day-ahead, intra-day, and real-time forecast errors of load are $\pm 20\%$, $\pm8\%$, and $\pm 3\%$, respectively\cite{20171225002}.
  \subsection{Results and Discussions}
  To verify the effectiveness of the proposed method, the upper-level problem is solved by the centralized rolling optimization method and the proposed scheme, separately. The total cost of IEH using method proposed in this paper is 247,185 yuan, while the centralized result is 247,173 yuan. Energy costs of each EH are listed in Table \ref{table_costs}. 
  Result of the propsoed method is verified to be very close to that of the centralized optimization method. Therefore, the method proposed in this paper is efficient in obtaining a rather optimal solution of the collaborative optimization in a distributed manner.

  \begin{table}[]
    \caption{Costs of 15 EHs under centralized(C) method and distributed(D) method proposed in this paper (yuan)}
    \label{table_costs}
    \begin{tabular}{@{}ccc|ccc|ccc@{}}
    \toprule
    No. & C\tnote{*}   & D\tnote{**}    & No. & C     & D     & No. & C     & D     \\ \midrule
    1   & 18504 & 18492 & 6   & 17982  & 18038  & 11  & 24717  & 24743  \\
    2   & 17064 & 17067 & 7   & 4860  & 4888  & 12  & 15425 & 15446 \\
    3   & 22886  & 22862  & 8   & 17499 & 17518 & 13  & 22858 & 22817 \\
    4   & 21209 & 21189 & 9   & 12660 & 12611 & 14  & 26982 & 26986 \\
    5   & 7928  & 7928  & 10  & 11766  & 11766  & 15  & 4833   & 4833   \\ \bottomrule
    \end{tabular}
  \end{table}

  Fig. \ref{fig_transformer} compares the main transformer power when using the two different schemes, the main grid real-time electricity price, the DA forecast price and the RT clearing electricity price curve. It can be seen that during 7:00-23:00, the real-time utility price is relatively high. The main transformer is not congested, and its power is basically the same in the two cases. On the contrary, the price is relatively low during 24:00-6:00. Congestion occurs, but both schemes ensure that the main transformer power is limited within the maximum capacity. Furthermore, it can also be seen from Fig. \ref{fig_transformer} that the clearing electricity price is higher than the forecast electricity price at 24:00. It is because that remaining unsatisfied shiftable loads are forced to be satisfied during the last scheduling period, thus lowering the EH's control flexibility. Besides, due to the existence of renewable energies and loads uncertainties, transformer power differences between these two schemes have been observed during some time periods such as 7:00 and 19:00.
  \begin{figure}
    \centering
    \includegraphics[width=0.48\textwidth]{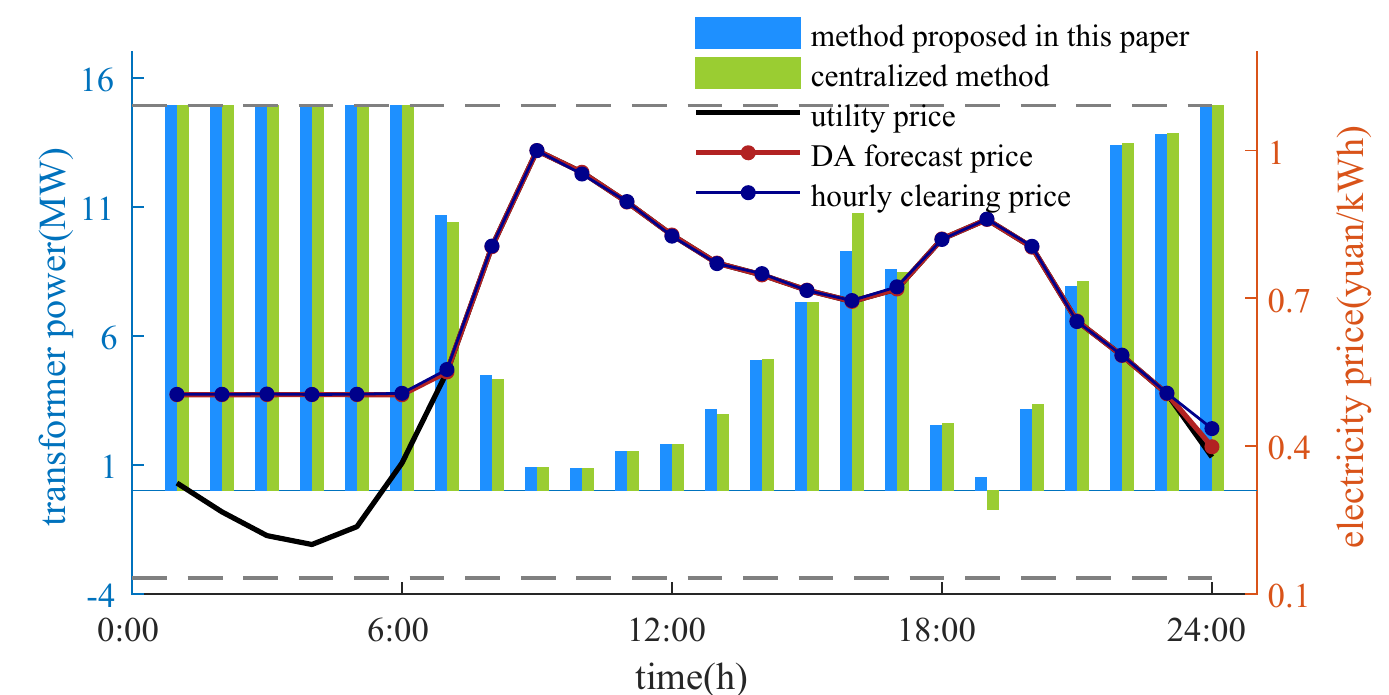}
    \caption{Transformer power of two methods and price curves }
    \label{fig_transformer}
  \end{figure}

  Simulation results of electric power of one EH is given in Fig. \ref{fig_EES}. As expected, conditions where \eqref{eqn_condition_constraint} is unsatisfied have not been witnessed in the simulation case. 
  \begin{figure}
    \centering
    \includegraphics[width=0.48\textwidth]{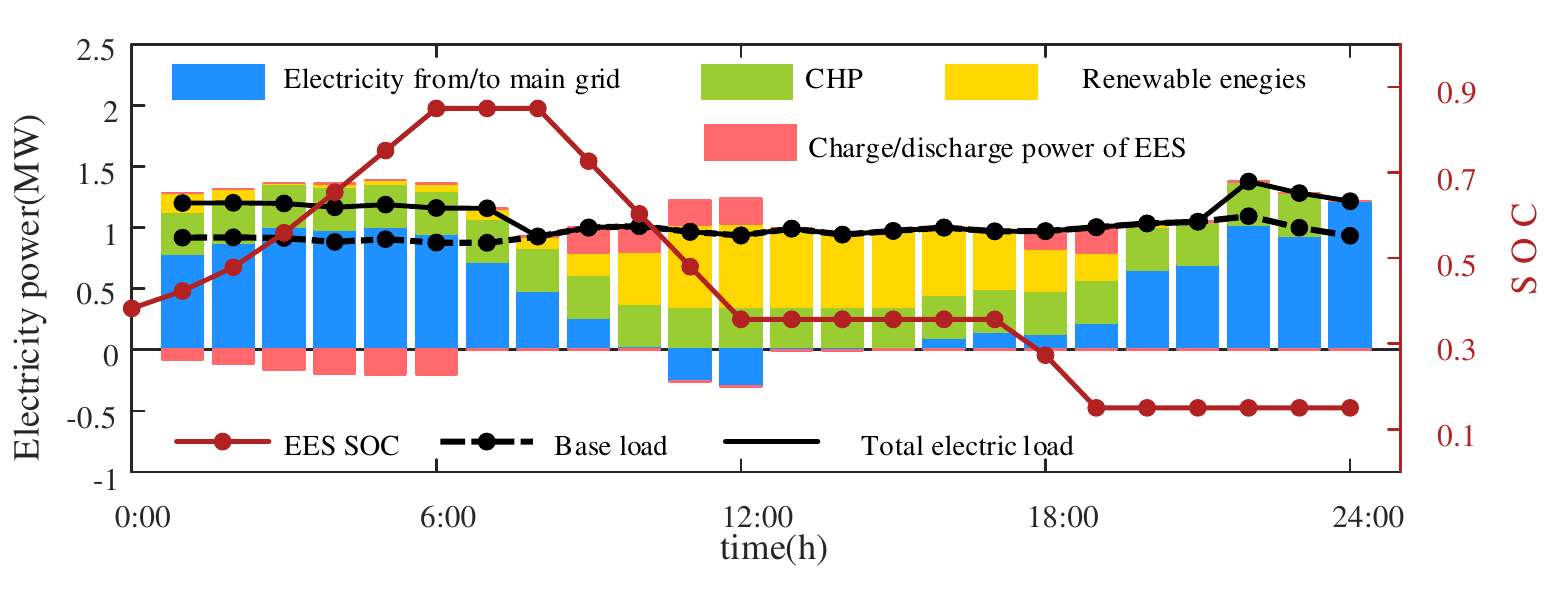}
    \caption{Simulation results of one EH's electric power}
    \label{fig_EES}
  \end{figure}

  This paper further simulates the efficiency of the proposed method in coordinating IEH of different scales. Results are shown in Table \ref{table_iteration}. The results show that by applying proposed method, the number of iterations during RT stage are greatly reduced compared with DA iteration. In addition, the number of iterations before and during the day will not increase significantly with the growth of EH numbers. These ensure that the computation complexity and iteration speed can meet the efficiency requirement for real-time control within the day, and system’s scalability is also maintained.
  
  \begin{table}[]
    \caption{Iterations of day-ahead and real-time optimization}
    \label{table_iteration}
    \centering
    \begin{tabular}{c|cccc}
      \toprule
     EH number & 10 & 20 & 50 & 100\\
     \toprule
     Day-ahead iteration & 140 & 117 & 78 & 92\\
     Real-time iteration & 9 & 9 & 9 & 9\\
     \bottomrule
    \end{tabular}
  \end{table}
  
  \section{Conclusion}
  Considering issues associated with information privacy and operation authority among different management sectors, this paper introduces TC method in managing IEH. Specially, a bi-level two-stage optimization framework is established, where each EH performs autonomous optimization in the lower level, and TC method is applied to realize collaborative energy management for IEH in the upper level. 

  In the autonomous optimization, this paper proposes an equivalent method of energy storage and provides its mathematical proof, in order to relax the nonlinear equality constraint and model the problem in a convex way. In the collaborative optimization, a two-stage procedure is designed to simplify the dual variables in the real-time stage, thus enabling the later utilization of a bisection method to solve the clearing price. Since bisection method excels in convergence performance, the optimization method can meet the real-time control requirements. The simulation case reaffirms that precise prices forecasted in the day-ahead stage can help the upper level to obtain a rather close result to the centralized method.

 

  \bibliographystyle{IEEEtran}
  \bibliography{references}

\begin{thebibliography}{1}
\providecommand{\url}[1]{#1}
\csname url@samestyle\endcsname
\providecommand{\newblock}{\relax}
\providecommand{\bibinfo}[2]{#2}
\providecommand{\BIBentrySTDinterwordspacing}{\spaceskip=0pt\relax}
\providecommand{\BIBentryALTinterwordstretchfactor}{4}
\providecommand{\BIBentryALTinterwordspacing}{\spaceskip=\fontdimen2\font plus
\BIBentryALTinterwordstretchfactor\fontdimen3\font minus
  \fontdimen4\font\relax}
\providecommand{\BIBforeignlanguage}[2]{{%
\expandafter\ifx\csname l@#1\endcsname\relax
\typeout{** WARNING: IEEEtran.bst: No hyphenation pattern has been}%
\typeout{** loaded for the language `#1'. Using the pattern for}%
\typeout{** the default language instead.}%
\else
\language=\csname l@#1\endcsname
\fi
#2}}
\providecommand{\BIBdecl}{\relax}
\BIBdecl

\bibitem{7842813}
E.~Dall'Anese, P.~Mancarella, and A.~Monti, ``Unlocking flexibility: Integrated
  optimization and control of multienergy systems,'' \emph{IEEE Power and
  Energy Magazine}, vol.~15, no.~1, pp. 43--52, Jan 2017.

\bibitem{EN2448}
Y.~Zhang, Y.~He, M.~Yan, C.~Guo, and Y.~Ding, ``Linearized stochastic
  scheduling of interconnected energy hubs considering integrated demand
  response and wind uncertainty,'' \emph{Energies}, vol.~11, p. 2448, 09 2018.

\bibitem{HU2017}
\BIBentryALTinterwordspacing
J.~Hu, G.~Yang, K.~Kok, Y.~Xue, and H.~W. Bindner, ``Transactive control: a
  framework for operating power systems characterized by high penetration of
  distributed energy resources,'' \emph{Journal of Modern Power Systems and
  Clean Energy}, vol.~5, no.~3, pp. 451--464, May 2017. [Online]. Available:
  \url{https://doi.org/10.1007/s40565-016-0228-1}
\BIBentrySTDinterwordspacing

\bibitem{8245724}
M.~Ji and P.~Zhang, ``Transactive control and coordination of multiple
  integrated energy systems,'' in \emph{2017 IEEE Conference on Energy Internet
  and Energy System Integration (EI2)}, Nov 2017, pp. 1--6.

\bibitem{MAJIDI2017157}
\BIBentryALTinterwordspacing
M.~Majidi, S.~Nojavan, and K.~Zare, ``A cost-emission framework for hub energy
  system under demand response program,'' \emph{Energy}, vol. 134, pp. 157 --
  166, 2017. [Online]. Available:
  \url{http://www.sciencedirect.com/science/article/pii/S0360544217309957}
\BIBentrySTDinterwordspacing

\bibitem{BECK2016331}
\BIBentryALTinterwordspacing
T.~Beck, H.~Kondziella, G.~Huard, and T.~Bruckner, ``Assessing the influence of
  the temporal resolution of electrical load and pv generation profiles on
  self-consumption and sizing of pv-battery systems,'' \emph{Applied Energy},
  vol. 173, pp. 331 -- 342, 2016. [Online]. Available:
  \url{http://www.sciencedirect.com/science/article/pii/S0306261916305104}
\BIBentrySTDinterwordspacing

\bibitem{8486723}
J.~Hu, G.~Yang, C.~Ziras, and K.~Kok, ``Aggregator operation in the balancing
  market through network-constrained transactive energy,'' \emph{IEEE
  Transactions on Power Systems}, pp. 1--1, 2018.

\bibitem{20171225002}
H.~Wang, Q.~Ai, L.~Gan, X.~Zhou, and F.~HU,
  ``\BIBforeignlanguage{Chinese}{Collaborative optimization of combined cooling
  heating and powersystem based on multi-scenario stochastic programming and
  model predictive control},'' \emph{\BIBforeignlanguage{Chinese}{Automation of
  Electric Power System}}, vol.~42, pp. 51--58, 2018. (in Chinese).

\end{thebibliography}
  
  
  \end{document}